\documentclass[reqno]{amsart}
\usepackage{amssymb, bm, amsmath, latexsym, mathabx, dsfont,tikz, float}
\usepackage{makecell, multirow, longtable}
\usepackage{enumitem}
\usepackage{cite}
\usepackage{tikz}

\renewcommand{\baselinestretch}{\baselinestretch}
\renewcommand{\baselinestretch}{1.1}
\numberwithin{equation}{section}

\newtheorem{thm}{Theorem}[section]
\newtheorem{lem}[thm]{Lemma}

\theoremstyle{definition}

\theoremstyle{remark}
\newtheorem{rmk}[thm]{Remark}

\newtheorem*{claim}{Claim}

\numberwithin{equation}{section}

\newcommand{\ra}{{\, \rightarrow \,}}

\newcommand{\z}{{\mathbb Z}}

\newcommand{\q}{{\mathbb Q}}

\newcommand{\n}{{\mathbb N}}
\newcommand{\e}{{\epsilon}}
\newcommand{\oo}{{\mathcal{O}}}
\newcommand{\ok}[1]{{\mathcal{O}_K^{#1}}}
\newcommand{\Mod}[1]{\ (\mathrm{mod}\ #1)}

\newcommand{\tr}{\operatorname{Tr}}
\newcommand{\well}{\widetilde{\ell}}

\begin{document}

	
	\author{Daejun Kim}
	\address{School of Mathematics, Korea Institute for Advanced Study,
		Seoul 02455, Republic of Korea}
	\email{dkim01@kias.re.kr}
	
	\thanks{This work of the first author was supported by a KIAS Individual Grant (MG085501) at Korea Institute for Advanced Study.}
	
	\author{Seok Hyeong Lee}
	\address{Center for Quantum Structures in Modules and Spaces, Seoul National University, Seoul 08826, Republic of Korea}
	\thanks{The second author was supported by the National Research Foundation of Korea(NRF) grant funded by the Korea government(MSIT) (No.2020R1A5A1016126).}
	\email{lshyeong@snu.ac.kr}

	
	
	\thanks{}
	
	
	\title[Lifting problem for universal quadratic forms]{Lifting problem for universal quadratic forms over totally real cubic number fields}

	\begin{abstract} 
		Lifting problem for universal quadratic forms asks for totally real number fields $K$ that admit a positive definite quadratic form with coefficients in $\mathbb{Z}$ that is universal over the ring of integers of $K$. In this paper, we show that $K=\q(\zeta_7+\zeta_7^{-1})$ is the only such totally real cubic field. Moreover, we show that there is no such biquadratic field.
	\end{abstract}
	
	\maketitle

	\section{Introduction}
	
	Representations of integers by a given quadratic form have a long and storied history. One particular interest has been universal quadratic forms, positive definite forms (with integer coefficients) that represent all positive integers. Lagrange provided the first example of a universal form by showing his Four Square Theorem, which says that every positive integer is a sum of at most four squares of integers. In other word, the form $x^2+y^2+z^2+w^2$ is universal over $\z$. The task of classifying all universal forms has been achieved by two significant theorems: $15$-theorem of Conway-Schneeberger \cite{B}, and $290$-theorem of Bhargava-Hanke \cite{BH}.
	
	The study of universal quadratic form has been extended to include totally real number fields $K$. In this setting, we consider positive definite quadratic forms having coefficients in the ring $\ok{}$ of integers of $K$. A such quadratic form is said to be universal over $\ok{}$ if it represents all totally positive elements of $\ok{}$. The results in the paper of Hsia, Kitaoka, Kneser \cite{HKK} imply that there exists a universal quadratic form over $\ok{}$ for any totally real number field $K$. Subsequently, numerous results concerning them have been studied, see, for example \cite{BK1, BK2, CLSTZ, CI, CKR, D1, D2, Ea, EK, Ka1, KS, KT, KY, KY2, KYZ, Ki1, Ki2, KKP, KTZ, Sa, Ya}. We refer the interested readers to a survey paper of Kala \cite{Ka2} for the recent developments on universal quadratic forms (and lattices) over the rings of integers of totally real number fields, and the references therein.

	In 1941, Maa$\ss$ \cite{M} proved that the sum of three squares is universal over $K=\q(\sqrt{5})$. Later, in 1945, Siegel \cite{Sie1} showed that the sum of any number of squares is universal over $\ok{}$ only if $K=\q,\q(\sqrt{5})$. This result immediately implies that any positive definite diagonal quadratic form with integer coefficients can be universal over $\ok{}$ only if $K=\q,\q(\sqrt{5})$.
	
	Motivated by these results, Kala and Yatsyna \cite{KY} recently studied the following ``Lifting problem": When is it possible for a $\z$-form, i.e., a positive definite quadratic form with coefficients in $\z$, to be universal over $\ok{}$? In the same paper, they proved that $K=\q(\sqrt{5})$ is the only real quadratic field that admits a $\z$-form which is universal over $\ok{}$. There are several $\z$-forms that are universal over $\mathcal{O}_{\q(\sqrt{5})}$ known, such as $x^2+y^2+z^2$ \cite{M}, $x^2+y^2+2z^2$ \cite{CKR}, $x^2+xy+y^2+z^2+zw+w^2$ \cite{D1}, $x^2+y^2+z^2+w^2+xy+xz+xw$ \cite{D2}. Furthermore, Kala and Yatsyna showed the following theorem regarding the existence of universal $\z$-forms over some number fields of small degree.
	\begin{thm}[{\cite[Theorem 1.2]{KY}}] \label{thm:KY}
		There does not exist a totally real number field $K$ of degree $1,2,3,4,5,$ or $7$ which has principal codifferent ideal and a universal $\z$-form defined over $\ok{}$, unless $K=\q, \q(\sqrt{5}),$ or $\q(\zeta_7+\zeta_7^{-1})$. The $\z$-form $x^2+y^2+z^2+w^2+xy+xz+xw$ is universal over $\q(\zeta_7+\zeta_7^{-1})$.
	\end{thm}
	
	\begin{rmk} It is well-known that a monogenic field (field such that $\ok{} = \z[\rho]$ for some $\rho \in \ok{}$) has principal codifferent ideal. Moreover, if the ideal class group $\textrm{Cl}_K$ of $\ok{}$ is isomorphic to $(\z/2\z)^r$ for some $r\in\n$, then the codifferent ideal is principal, due to Hecke's theorem stating that the different ideal is always square in $\textrm{Cl}_K$ (see \cite{Ar} for the proof). 
		
		On the other hand, there are quite many totally real cubic fields with non-principal codifferent ideal. We downloaded from LMFDB \cite{LMFDB} the list of totally real cubic fields of small class numbers in Magma format, and used Magma to check whether their different ideals are principal or not (hence their codifferent ideals). Let $K$ stands for totally real cubic fields for a moment. Among first $1000$ fields $K$ (ordered by the size of the discriminant) with $\textrm{Cl}_K\cong \z/3\z$ out of $27663$ in the list, $525$ had non-principal codifferent ideals. Out of all the $2073$ or $1150$ fields $K$ in LMFDB with $\textrm{Cl}_K\cong \z/4\z$ or $\z/5\z$ respectively, $545$ or $737$ had non-principal codifferent ideals, respectively.
	\end{rmk}
	
	As described in \cite{KY}, the assumptions of Theorem \ref{thm:KY} on the totally real number fields under consideration come from the tools that they used: Siegel's formula for the value of Dedekind zeta function $\zeta_K(s)$ at $s=-1$ \cite{Sie2,Zag}. This formula expresses the value in terms of elements in the codifferent ideal of small trace. In the case of degree $d=2,3,4,5,$ or $7$, only trace $1$ elements appear in the formula. Moreover, if the codifferent is principal, then computing $\zeta_K(-1)$ boils down to counting the number of trace $1$ elements in the codifferent, which impose a bound in terms of $d$. This bound, together with the functional equation of $\zeta_K(s)$ relating $\zeta_K(-1)$ and $\zeta_K(2)$, provides an upper bound for the discriminant of $K$. This reduces the problem to examining only finitely many $K$ for each degree. 
	
	It is natural to ask whether there exists a $\z$-form universal over $\ok{}$ for a number field $K$ not covered by Theorem \ref{thm:KY}. Indeed, Kala and Yatsyna \cite{KY} also conjectured that there may be no number field $K$ with a universal $\z$-form over $\ok{}$ except for $K=\q, \q(\sqrt{5})$ and $\q(\zeta_7+\zeta_7^{-1})$. Note that $\q(\zeta_7+\zeta_7^{-1})$ is a totally real cubic field, i.e., a number field of degree $d=3$. 
	
	In this paper, we affirmatively answer this conjecture for the totally real cubic fields or the real biquadratic fields by showing the following theorem.
	
	\begin{thm}\label{mainthm}
		There does not exist a $\z$-form that is universal over $\ok{}$ of a totally real cubic number field or a real biquadratic field $K$, unless $K=\q(\zeta_7+\zeta_7^{-1})$.
	\end{thm}
	
	We prove this by Theorem \ref{thm:z-formcubicfield} for the real cubic field case and by Theorem \ref{thm:z-formbiquadfield} for the biquadratic field case. Note that Theorem \ref{mainthm} does not require the codifferent ideal of $K$ to be principal, which is a restriction in Theorem \ref{thm:KY} for the number fields under consideration.

	The question whether there exists a universal $\z$-form over a totally real number field of degree $\ge 4$ other than biquadratic fields still remains open. Thanks to Kala and Yatsyna who studied the so-called ``weak lifting problem" and proved some finiteness results in more general setting (see \cite[Theorem 2]{KY2} for the general statement), we know that for $d,r\in\n$, there are only finitely many totally real number fields $K$ of degree $d=[K:\q]$ such that there is a $\z$-form of rank $r$ that is universal over $\mathcal{O}_K$ (see \cite[Corollary 9]{KY2}).
	
	Let us briefly describe the tools we used and summarize our strategy. We approach the problem using geometry of numbers. Let $K$ be a totally real number field of degree $d$ over $\q$, and let $\rho_1=\operatorname{id},\cdots, \rho_d : K \ra\mathbb{R}$ be the real embeddings of $K$. Define a map $\rho : K \ra \mathbb{R}^d$ by
	\[
	\rho(\alpha):= (\rho_1(\alpha), \rho_2(\alpha), \ldots, \rho_d(\alpha))
	\] 
	for any $\alpha \in K$. Then the image $\rho(\ok{})$ of the ring $\ok{}$ is a lattice in $\mathbb{R}^d$. Note that the lattice $\rho(\ok{})$ is identical to the $\z$-module $\ok{}$ equipped with a non-degenerate symmetric bilinear form defined by the trace map. Hence the image $\rho(\ok{\vee})$ of the codifferent $\ok{\vee}$ of $\ok{}$ is indeed the dual lattice of $\rho(\ok{})$ (see Section \ref{subsec:latticeO_K} for details).
	
	We first extend a useful lemma shown in \cite[Lemma 4.6]{KY} for totally positive elements in $\ok{}$ to arbitrary elements in $\ok{}$ in Lemma \ref{lem:genlem4.6ofKY}. This can be interpreted geometrically such that if $K$ admits a universal $\z$-form, then any $\alpha\in\ok{}$ that satisfies the following geometric condition should be a unit: there is a hyperplane in $\mathbb{R}^d$ passing through $\rho(\alpha)$ and orthogonal to $\rho(\delta)$ for some $\delta\in\ok{\vee}$, such that all points $\beta \in \ok{}$ with the same signature as $\alpha$ are on the same side (see Section \ref{subsec:notations} for the definition of the signature map). However, geometry of $\ok{}$ shows that, except for finitely many cubic fields $K$, there is an $\alpha$ that is not a unit but satisfies the geometric condition described above (i.e. such $\delta$ exists), which proves Theorem \ref{mainthm}.
	
	For the cubic case, we focus on the fact that elements of the dual lattice are in one-to-one correspondence with linear functionals defined on the original lattice. We try to find $\delta\in\ok{\vee}$ that satisfies the above condition by defining a special linear functional on $\ok{}$ (see Lemma \ref{lem:bij-codiff-linfunctional}). This enables us to show, in Lemma \ref{lem:pointsonlineareunits}, the existence of a line $\ell$ in $\mathbb{R}^3$ that is parallel to $\rho(1)$ and passes through a point $\rho(\beta)$ in $\rho(\ok{})$, such that every point of $\rho(\ok{})$ lying on the segment formed by the intersection of $\ell$ and one of the octants of $\mathbb{R}^3$ must be a unit. This gives certain bounds for $\rho_1(\beta),\rho_2(\beta),\rho_3(\beta)$ which allow us to limit the potential minimal polynomials of $\beta$ to only a few possibilities. Finally, among such finite candidates of $K = \q(\beta)$, we show $\q(\zeta_7+\zeta_7^{-1})$ is the only field posing no contradictions to Lemma \ref{lem:pointsonlineareunits}.
	
	For the biquadratic case, we can explicitly write an integral basis of $\ok{}$ and hence that of $\ok{\vee}$. Using this information, we can explicitly exhibit a nonunit $\alpha\in\ok{}$ and $\delta\in\ok{\vee}$ satisfying the conditions of Lemma \ref{lem:genlem4.6ofKY}. It will immediately follow that $K$ does not admit a universal $\z$-form.

	The paper is organized as follows. In Section \ref{sec:prelim}, we set up the notation, recall the theory of quadratic lattices, and introduce some useful lemmas. In Section \ref{sec:quadratic}, we provide a geometric proof for the real quadratic field case, giving motivation for our geometric idea. In Section \ref{sec:cubic}, we prove Theorem \ref{thm:z-formcubicfield}, which proves Theorem \ref{mainthm} for the real cubic field case. Finally, in Section \ref{sec:biquadratic}, we provide the proof for the real biquadratic field case in Theorem \ref{thm:z-formbiquadfield}, completing the proof of Theorem \ref{mainthm}.

	\section{Preliminaries}\label{sec:prelim}
	
	\subsection{Notations and terminologies}\label{subsec:notations}
	Throughout the paper, $K$ will denote a totally real number field of degree $d=[K:\q]$ over $\q$.
	Let $\rho_1=\operatorname{id},\cdots, \rho_d : K \ra \mathbb{R}$ be the real embeddings of $K$, and define a map $\rho : K \ra \mathbb{R}^d$ by
	\begin{equation}\label{def:rho}
		\rho(\alpha):= (\rho_1(\alpha), \rho_2(\alpha), \ldots, \rho_d(\alpha))
	\end{equation}
	for any $\alpha \in K$. The norm of $\alpha\in K$ is $N(\alpha):=\rho_1(\alpha)\cdots \rho_d(\alpha)$ and its trace is $\tr(\alpha):=\rho_1(\alpha)+\cdots \rho_d(\alpha)$. Let $\operatorname{Sgn}:K^\times \ra \{-1,1\}^d$ be the signature map, defined by $\operatorname{Sgn}(\alpha):=(\e_1,\ldots,\e_d)$, where $\e_i$ is the sign of $\rho_i(\alpha)$ for any $\alpha \in K^\times$. The signature map is a homomorphism, and the elements of the kernel of the signature map are called totally positive numbers. Denote totally positive elements of $\ok{}$ by
	\[
	\ok{+}:=\{\alpha \in\oo_K : \rho_i(\alpha)>0 \text{ for all }i\},
	\]
	and denote the codifferent of $\ok{}$ by $\ok{\vee}:=\{\beta\in K : \operatorname{Tr}(\beta\ok{})\subseteq \z\}$. For $\epsilon=(\e_1,\ldots,\e_d)\in \{-1,1\}^d$, we further denote that
	\[
	\ok{\epsilon}:=\{\alpha \in\ok{} : \operatorname{Sgn}(\alpha)=\epsilon\}, \quad \ok{\vee,\epsilon} := \{\delta \in \ok{\vee}: \operatorname{Sgn}(\delta)=\epsilon\}.
	\]
	Note that $\ok{+}=\ok{(1,\ldots,1)}$ by their definitions, and we simply write $\ok{\vee,+}=\ok{\vee,(1,\ldots,1)}$.
	
	We will work with a positive definite quadratic form with integer coefficients, that is, a homogeneous polynomial $Q\in\z[x_1,\ldots,x_r]$ of degree two of the form
	\[
	Q(x_1,x_2,\ldots,x_r)=\sum_{1\le i\le j\le r} a_{ij} x_ix_j \text{ with } a_{ij}\in\z
	\]
	that satisfies $Q(x)> 0$ for all $x\in\mathbb{R}^r\setminus\{0\}$. We will simply call such quadratic form as $\z$-form. One may also consider a quadratic form over $\ok{}$, that is a quadratic form $Q(x)=\sum_{i\le j} a_{ij}x_ix_j$ with $a_{ij}\in\ok{}$. Such a form $Q$ is said to be {\em totally positive} if $Q(v)\in\ok{+}$ for all $v\in \ok{r}\setminus\{0\}$. One may consider a $\z$-form $Q$ as a totally positive quadratic form over $\ok{}$. We say a $\z$-form $Q$ {\em represents} $\alpha\in\ok{+}$ if $Q(v)=\alpha$ for some $v\in\ok{r}$, and say a $\z$-form $Q$ is {\em universal} over $\ok{}$ if it represents any $\alpha\in\ok{+}$.

	Now we introduce the geometric language of quadratic spaces and lattices. A {\em quadratic space} over $\q$ is a vector space $V$ over $\q$ equipped with a non-degenerate symmetric bilinear form $B:V\times V \ra \q$. We sometimes denote the quadratic space as a pair $(V,B)$. We say $V$ is positive definite if the associated quadratic form $Q(v)=B(v,v)$ is positive definite, that is, $Q(v)>0$ for all $v\in V\setminus\{0\}$. A {\em quadratic $\z$-lattice} $L$ on $V$ is a finitely generated $\z$-module such that $\q L=V$. Since the ring $\z$ of rational integers is a principal ideal domain, every $\z$-lattice is a free $\z$-module. Write $L=\z x_1 + \cdots + \z x_r$ for a basis $x_1,\ldots ,x_r$ of $V$. The {\em dual} of $L$ is defined by
	\[
	L^{\#}:=\left\{ x\in V : B(x,L)\subseteq \z \right\}.
	\]
	and the {\em dual basis} $y_1, \ldots, y_r$ of a basis $x_1, \ldots, x_r$ is defined as the list of elements $y_j \in V$ satisfying $B(x_i, y_j) = \delta_{ij}$ for $1 \le i,j \le r$. It is well known that $L^{\#}$ is also a lattice on $V$ and any dual basis becomes its $\z$-basis. We say $L$ is {\em integral} if $B(L,L)\subseteq \z$. 
	
	Any unexplained notation and terminology on quadratic lattices can be found in \cite{OM2}.

	\subsection{Lattice structure on $\ok{}$ and $\rho(\ok{})$}\label{subsec:latticeO_K}
	Let us consider a map $B_{\tr}:K \times K \rightarrow \q$ defined by 
	\[
	B_{\tr}(\alpha,\beta):=\tr(\alpha\beta) \quad \text{ for any }\alpha, \beta \in K.
	\] 
	One may easily show that $B_{\tr}$ is a non-degenerate symmetric bilinear form defined on $K$, hence we may consider $K$ as a quadratic space over $\q$. Let $\alpha_1,\ldots,\alpha_d$ be an integral basis of $\ok{}$, that is, $\ok{}=\z\alpha_1 + \cdots +\z\alpha_d$. Since $\tr(\ok{})\subseteq \z$, the ring of integers $\ok{}$ is a integral $\z$-lattice on $K$. Moreover, it follows that  $\ok{\#}=\ok{\vee}$, from the definitions of the codifferent ideal and the dual of a lattice.
	
	Recall the map $\rho:K\ra \mathbb{R}^d$ defined in \eqref{def:rho} and let $\cdot$ denote the standard inner product on $\mathbb{R}^d$. Noting for any $\alpha,\beta\in K$ that 
	\[
	\tr(\alpha\beta)=\sum_{i=1}^d \rho_i(\alpha\beta)=\sum_{i=1}^d \rho_i(\alpha)\rho_i(\beta)=\rho(\alpha)\cdot \rho(\beta),
	\]
	we may identify the $\z$-lattice $\ok{}$ on $(K,B_{\tr})$ with the $\z$-lattice $\rho(\ok{})$ in the Euclidean space $(\mathbb{R}^d,\cdot )$. Thus one may understand the $\z$-lattice $(\ok{},B_{\tr})$ geometrically by looking it as a lattice $\rho(\ok{})$ in the Euclidean space.
	
	Let $H$ be the hyperplane in $\mathbb{R}^d$ orthogonal to $\rho(1)=(1,\ldots, 1)$, and let $p:\mathbb{R}^d \ra H$ be the orthogonal projection onto $H$. Note that the image $\Lambda:=p(\rho(\ok{}))$ is a lattice on $H$ since $\rho(1)$ is a primitive vector of $\rho(\ok{})$. The following lemma describes how linear functionals defined on $\ok{}$ or those on $\Lambda$ are linked with the elements of $\ok{\vee}$.
	
	\begin{lem}\label{lem:bij-codiff-linfunctional}
		For $\delta\in\ok{\vee}$, define the linear functional $f_\delta : \ok{} \ra \z$ by $f_\delta(\alpha):=\tr(\delta\alpha)$ for any $\alpha\in \ok{}$. Then, the map 
		\[
		\ok{\vee} \ra \{\text{linear functionals } f:\ok{}\ra \z \},\quad \delta \mapsto f_\delta
		\]
		is a bijection. Moreover, if $g:\Lambda\ra \z$ is a linear functional on $\Lambda:=p(\rho(\ok{}))$, then the linear functional $f(g):=g\circ p \circ\rho:\ok{}\ra\z$ corresponds to $\delta\in \ok{\vee}$ such that $\tr(\delta)=0$.
	\end{lem}
	\begin{proof}
		The first bijection follows from the fact that $\ok{\#}=\ok{\vee}$. If $\delta\in\ok{\vee}$ satisfies $f(g)=f_\delta$, then 
		\[
		\tr(\delta)=\tr(\delta\cdot 1)=f_\delta(1)=f(g)(1)=g(0)=0.
		\]
		This implies the lemma.
	\end{proof}

	\subsection{Some lemmas} We introduce some useful lemmas giving necessary conditions for $K$ which admits a $\z$-form universal over $\ok{}$. The first lemma is due to \cite[Corollary 4.5]{KY}.
	
	\begin{lem}[{\cite[Corollary 4.5]{KY}}]\label{lem:cor4.5ofKY}
		Let $K$ be a totally real number field of degree $d\le 43$. Assume that there is a $\z$-form that is universal over $\ok{}$. Then every totally positive unit is a square in $\ok{}$. Hence there is a unit of every signature in $\ok{}$.
	\end{lem}
	
	We should remark that the condition $d\le 43$ on the degree of $K$ comes from the tool they used; the concept of lattices of $E$-type, which was first introduced by Kitaoka \cite{Kit1}. Consider an integral basis $\alpha_1,\ldots,\alpha_d$ of $\ok{}$. For $\delta\in\ok{\vee,+}$, define the twisted trace form $\tr_\delta: K \ra \q$ by $\tr_\delta(x):=\tr\left(\delta x^2\right)$ for any $x\in K$. Then $\tr_\delta(\ok{})\subseteq \z_{\ge0}$ and furthermore $(\ok{},B_\delta)$ becomes a positive definite integral $\z$-lattice of rank $d$, where $B_\delta:K\times K\ra\q$ is defined by $B_\delta(x,y):=\frac{1}{2}\left(\tr_\delta(x+y)-\tr_\delta(x)-\tr_\delta(y)\right)=\tr(\delta xy)$ for any $x,y\in K$. The exact condition that is required for $K$ in order for Lemma \ref{lem:cor4.5ofKY} to be satisfied is that $(\ok{},B_\delta)$ is a lattice of $E$-type for any $\delta\in \ok{\vee,+}$. This is true if $d\le43$ since every $\z$-lattice of rank at most $43$ is of $E$-type by \cite[Theorem 7.1.1]{Kit2}. We refer the readers to Chapter 7 of \cite{Kit2} for  and Section 4 of \cite{KY} for more details on lattice of $E$-type.
	
	Now we describe an extended version of \cite[Lemma 4.6]{KY}, which will play a crucial role together with the geometric interpretation of the $\z$-lattice $\ok{}$ in proving our main results.
	\begin{lem}\label{lem:genlem4.6ofKY}
		Let $K$ be a totally real number field of degree $d\le 43$ and let $\epsilon\in\{1,-1\}^d$. Assume that there is a $\z$-form that is universal over $\ok{}$ and let $\alpha\in \ok{\epsilon}$. If there is a $\delta\in \ok{\vee,\epsilon}$ such that $\tr(\delta\alpha)\le \tr(\delta\beta)$ for any $\beta\in \ok{\epsilon}$, then $\alpha$ is a unit in $\ok{}$. In particular, if $\tr(\delta\alpha)=1$ for some $\delta\in \ok{\vee,\epsilon}$, then $\alpha$ is a unit in $\ok{}$.
	\end{lem}
	\begin{proof}
		The case when $\epsilon=(1,\ldots,1)$ is just \cite[Lemma 4.6]{KY} itself. To prove the lemma for general $\epsilon\in\{1,-1\}^d$, take a unit $u$ in $\ok{\epsilon}$ whose existence is ensured by Lemma \ref{lem:cor4.5ofKY}. Then $\delta':=\delta u^{-1}\in \ok{\vee,+}$ and $\alpha':=u\alpha \in \ok{+}=\ok{(1,\ldots,1)}$ satisfies the assumption of the lemma for the case when $\epsilon=(1,\ldots,1)$; for any $\beta'\in\ok{+}$, noting that $u^{-1}\beta\in\ok{\epsilon}$, we have $\tr(\delta'\alpha')=\tr(\delta\alpha)\le \tr(\delta\beta)=\tr(\delta'\beta')$. Therefore, $\alpha'=u\alpha$ is a unit, hence so is $\alpha$.  
		
		To prove the ``in particular" part, let $\delta\in \ok{\vee,\epsilon}$ and $\alpha\in \ok{\epsilon}$ be such that $\tr(\delta\alpha)=1$. Noting that $\delta\beta\in\ok{+}$ for any $\beta  \in\ok{\epsilon}$, we have $\tr(\delta\beta)\in\z_{>0}$. Hence $\tr(\delta\alpha)=1\le \tr(\delta\beta)$ for any $\beta \in \ok{\epsilon}$, so $\alpha$ is a unit.
	\end{proof}
	
	\section{When $d=2$: real quadratic number field case}\label{sec:quadratic}
	
	In this section, we provide another proof of \cite[Theorem 1.1]{KY} to give a motivation for the proof in the cubic field case. The proof mainly uses Lemma \ref{lem:genlem4.6ofKY} together with a geometric interpretation of the $\z$-lattice $\rho(\ok{})$ in $2$-dimensional Euclidean space.
	\begin{thm}[Theorem 1.1 of \cite{KY}]\label{thm:z-formquadfield}
		There does not exist a $\z$-form that is universal over a real quadratic number field $K$, unless $K=\q(\sqrt{5})$.
	\end{thm}
	
	\subsection{An algebraic proof} Let $K=\q(\sqrt{D})$ be a real quadratic number field, where $D\ge 2$ is a squarefree integer, and let $\ok{}=\z 1 + \z \tau$, where $\tau = \sqrt{D}$ if $D\equiv 2,3 \Mod{4}$ and $\tau=\frac{1+\sqrt{D}}{2}$ if $D\equiv 1\Mod{4}$. Consider the linear functional $f:\ok{} \ra \z$ defined by $f(1)=0$ and $f(\tau)=1$.
	By Lemma \ref{lem:bij-codiff-linfunctional}, there is $\delta_f\in \ok{\vee}$ such that $\tr(\delta_f)=\tr(\delta_f 1)=f(1)=0$ and $\tr(\delta_f \tau)=f(\tau)=1$. Since $\rho_1(\tau)=\tau>0$ and $\rho_2(\tau)<0$, one may observe that $\rho_1(\delta_f)=-\rho_2(\delta_f)>0$. Hence we have $\delta_f \in \ok{\vee,(1,-1)}$.
	
	On the other hand, let $\alpha(a):=a+\tau\in\ok{}$ with $a\in\z$. Note that $\tr(\delta_f \alpha(a)) = f(\alpha(a))=f(a\cdot 1+\tau) = 1$. Thus, if $\alpha(a) \in \ok{(1,-1)}$, then $\alpha$ is a unit by Lemma \ref{lem:genlem4.6ofKY}, and hence $N(\alpha(a))=-1$. Noting that 
	\[
	N(\alpha(a))=N(a+\tau)=\begin{cases}
		a^2-D & \text{if } D\equiv 2,3\Mod{4},\\
		\left(a+\frac{1}{2}\right)^2-\frac{D}{4} & \text{if } D\equiv 1\Mod{4},
	\end{cases}
	\]
	there are at most two $a\in\z$ satisfying $N(\alpha(a))=-1$. Thus, there should be at most two $\alpha(a)$ with $\alpha(a)\in\ok{(1,-1)}$. One may easily check that $\alpha(-1),\alpha(0),\alpha(1)\in\ok{(1,-1)}$ if $D\neq 5$. This proves the statement. Indeed, in the original proof in \cite{KY}, it is used that $N(\alpha(0))=-1$ implies $D=5$ to conclude the proof.
	
	\subsection{Geometric interpretation} Now, let us identify $\ok{}$ with the $\z$-lattice $\rho(\ok{})\subseteq\mathbb{R}^2$, and interpret the above argument geometrically in $\mathbb{R}^2$. For $x\in \mathbb{R}^2$, let $\ell_x$ denote the line in $\mathbb{R}^2$ parallel to $\rho(1)=(1,1)$ passing through $x$. Then the linear functional $f:\ok{}\ra \z$ defined above corresponds to the linear functional  $f\circ\rho^{-1}:\rho(\ok{}) \ra \z$, and it can be characterized as one satisfying $f\circ\rho^{-1}(1,1)=0$, $f\circ\rho^{-1}(x_1,x_2)=1$ for all $\alpha\in\ok{}$ with $\rho(\alpha) \in \ell_{\rho(\tau)}$. Since $\rho(\ok{})^{\#}=\rho(\ok{\vee})$, the element $\delta_f\in\ok{\vee}$ defined above satisfies $\rho(\delta_f)\cdot (1,1)=0$, and $\rho(\delta_f)\cdot \rho(\alpha)=1$ for all $\alpha\in \ok{}$ with $\rho(\alpha)\in\ell_{\rho(\tau)}$.
	Then Lemma \ref{lem:genlem4.6ofKY} implies that if $\alpha\in\ok{}$ satisfies $\rho(\alpha)\in \ell_{\rho(\tau)}\cap \{(x_1,x_2):x_1>0,\ x_2<0\}$, then $\rho_1(\alpha)\rho_2(\alpha)=-1$. Note that 
	\[
	|\ell_{\rho(\tau)}\cap \{(x_1,x_2)\in\mathbb{R}^2: x_1>0,\ x_2<0,\ x_1x_2=-1\}|\le 2,
	\]
	and on the line $\ell_{\rho(\tau)}$, points of $\rho(\ok{})$ are placed every $||\rho(1)||=\sqrt{2}$ distance, where $||\cdot ||$ denotes the Euclidean norm. Hence the length of the segment $\ell_{\rho(\tau)}^{(1,-1)}:=\ell_{\rho(\tau)}\cap \{(x_1,x_2)\in\mathbb{R}^2: x_1\ge0,\ x_2\le0\}$ should be less than $3\sqrt{2}$, otherwise it contains at least three points of $\rho(\ok{})$. Since the length of $\ell_{\rho(\tau)}$ is given as $\sqrt{2} (\rho_1(\tau) - \rho_2(\tau))$, we have $\rho_1(\tau)-\rho_2(\tau) < 3$ and this is possible only when $D=2$ or $D=5$. However, when $D=2$, the following three points of $\rho(\ok{})$
	\[
	\rho(\alpha(-1))=(-1+\sqrt{2},-1-\sqrt{2}),\ \rho(\alpha(0))=(\sqrt{2},-\sqrt{2}),\ \rho(\alpha(1))=(1+\sqrt{2},1-\sqrt{2})
	\]
	lie on the segment $\ell_{\rho(\tau)}^{(1,-1)}$, which is a contradiction.
	
	\section{When $d=3$: real cubic number field case}\label{sec:cubic}
	
	In this section, we prove the following theorem which completely answers the Lifting problem for totally real cubic number field case. Throughout this section, $K$ always stands for a totally real number field of degree $[K:\q]=3$.
	\begin{thm}\label{thm:z-formcubicfield}
		There does not exist a $\z$-form that is universal over a totally real cubic number field $K$, unless $K=\q(\zeta_7+\zeta_7^{-1})$.
	\end{thm}

	For $x=(x_1,x_2,x_3)\in \mathbb{R}^3$, let $\ell_x$ denote the line in $\mathbb{R}^3$ parallel to $\rho(1)=(1,1,1)$ passing through $x$. Note that the three planes $x_i=0$ for $i=1,2,3$ split $\ell_x$ into at most four pieces, at least two of which are rays. Precisely, for indices $i,j,k$ with $\{i,j,k\}=\{1,2,3\}$ satisfying $x_k\le x_j\le x_i$, define
	\[
	\begin{aligned}
		\ell_{x,+}&:=\{ (y_1,y_2,y_3) : y_i=t+x_i,\ y_j=t+x_j,\ y_k=t+x_k,\ t>-x_k\}\\
		\ell_{x,1}&:=\{ (y_1,y_2,y_3) : y_i=t+x_i,\ y_j=t+x_j,\ y_k=t+x_k,\ -x_j<t<-x_k\}\\
		\ell_{x,2}&:=\{ (y_1,y_2,y_3) : y_i=t+x_i,\ y_j=t+x_j,\ y_k=t+x_k,\ -x_i<t<-x_j\}\\
		\ell_{x,-}&:=\{ (y_1,y_2,y_3) : y_i=t+x_i,\ y_j=t+x_j,\ y_k=t+x_k,\ t<-x_i\}
	\end{aligned}
	\]
	If $x_1, x_2, x_3$ are all distinct, then $\ell_{x,1}$ and $\ell_{x,2}$ become finite segments, and $\ell_{x,+}$ and $\ell_{x,-}$ become rays. Define $\well_x = \ell_{x,1} \cup \ell_{x,2}$.
	
	\begin{lem}\label{lem:pointsonlineareunits}
		There exists $\beta\in\ok{}\setminus\z$ whose corresponding line $\ell_{\rho(\beta)}$ satisfies the following property: if $\alpha\in\ok{}$ satisfies $\rho(\alpha)\in\well_{\rho(\beta)}$, then $\alpha$ is a unit.
	\end{lem}
	
	Before we prove Lemma \ref{lem:pointsonlineareunits}, we first describe the proof of Theorem \ref{thm:z-formcubicfield} using the lemma.
	
	\begin{proof}[Proof of Theorem \ref{thm:z-formcubicfield}]
		Let us take $\beta\in\ok{}\setminus\z$ satisfying the property in Lemma \ref{lem:pointsonlineareunits}. We first note that $\beta \notin \z$ implies $\rho_1(\beta), \rho_2(\beta), \rho_3(\beta)$ are all distinct and none of them are integers. We observe that for $i\neq j$, the set $\{ x \in K : \rho_i(x) = \rho_j(x) \}$ is a proper subfield of cubic extension $K$, so it should be $\q$. Thus if $\rho_i(\beta)=\rho_j(\beta)$, then $\beta \in \q\cap \ok{}=\z$, which is a contradiction.
		
		Let $i,j,k$ be indices with $\{i,j,k\}=\{1,2,3\}$ such that $\rho_k(\beta)< \rho_j(\beta) < \rho_i(\beta)$. By shifting $\beta$ by an integer if necessary, we may assume that $-1 < \rho_k(\beta)<0$. If $\rho_j(\beta)>2$, then we have
		\[
		\rho_j(\beta-1)>1, \ \rho_i(\beta-1) > \rho_j(\beta-1) > 1 \text{ and } \rho_k(\beta-1)<-1.
		\]
		So we have $\rho(\beta-1) \in \ell_{\rho(\beta),1}$ but $|N(\beta-1)|=|\rho_i(\beta-1)\rho_j(\beta-1)\rho_k(\beta-1)|>1$, thus it contradicts Lemma \ref{lem:pointsonlineareunits}. Therefore we should have $\rho_j(\beta)<2$.	Similarly if $\rho_i(\beta)>4$ then
		\[
		\rho_i(\beta-3)>1, \ \rho_j(\beta-3)<-1 \text{ and } \rho_k(\beta-3)<-3,
		\] 
		so $\rho(\beta-3) \in \ell_{\rho(\beta),2}$ but $|N(\beta-3)|>3$ contradicts Lemma \ref{lem:pointsonlineareunits}. Thus $\rho_i(\beta)<4$. Meanwhile, if $\rho_i(\beta)<1$ then $-1 < \rho_k(\beta) < \rho_j(\beta)<\rho_i(\beta)<1$, so $|N(\beta)|<1$ contradicts that $\beta \in \ok{}$. Thus $\rho_i(\beta)>1$. Collecting all those bounds gives
		\begin{equation}\label{eqn:rhobetasbound}
			-1 < \rho_k(\beta)<0, \quad -1 < \rho_j(\beta)<2, \quad 1 < \rho_i(\beta)<4.
		\end{equation}
		Write the minimal polynomial of $\beta$ as
		\[
		m_{\beta}(x) = (x- \rho_i(\beta)) (x - \rho_j(\beta)) (x- \rho_k(\beta)) = x^3 - c_1(\beta) x^2 + c_2 (\beta) x - c_3(\beta).
		\]
		As $\beta \in \ok{}$, $m_{\beta}(s)$ is an integral polynomial. From $c_1(\beta) = \rho_i(\beta)+\rho_j(\beta)+\rho_k(\beta)$, we can obtain $-1 < c_1(\beta)<6$ hence $0 \le c_1(\beta) \le 5$ since $c_1(\beta)\in\z$. As $\rho_i(\beta)>0$ and $\rho_k(\beta)<0$, $\rho(\beta) \in \well_{\rho(\beta)}$ so $\beta$ is a unit and therefore $|c_3(\beta)| = |N(\beta)|=1$ by Lemma \ref{lem:pointsonlineareunits}. Observing that $\rho_i(\beta-1)>0$ and $\rho_k(\beta-1)<-1$, we also have $\rho(\beta-1) \in \well_{\rho(\beta)}$ so $|N(\beta-1)| = |m_{\beta}(1)| = 1$ since $\beta-1$ should be a unit by Lemma \ref{lem:pointsonlineareunits}. Now we consider the following $24$ candidates for $m_\beta(s)$ whose coefficients satisfy
		\[
		0 \le c_1(\beta) \le 5, \quad c_3(\beta) \in \{1, -1\}, \quad m_{\beta}(1) = 1 - c_1(\beta) + c_2(\beta) - c_3(\beta)  \in \{1,-1\}
		\]
		Among these $24$ cubics, one may check that those satisfying the conditions \eqref{eqn:rhobetasbound} as well are given as
		\[
		x^3 - 2x^2 - x +1, \quad x^3 - 3x^2 + 1, \quad x^3 - 4x^2 + x + 1, \quad x^3 - 4x^2 + 3x+1, \quad x^3 - 5x^2 + 4x + 1.
		\]
		One may check that the second, the third and the fifth cubics cannot be happened as they contradict Lemma \ref{lem:pointsonlineareunits}, while the others give possible real cubic field $K=\q(\beta)$ by noting that $\q(\zeta_7+\zeta_7^{-1})$ is Galois over $\q$ and the following:
		\begin{itemize}
			\item $x^3 - 2x^2 - x +1$: it has $\zeta_7+\zeta_7^{-1}+1$ as a root, so $K=\q(\beta)=\q(\zeta_7+\zeta_7^{-1})$.
			
			\item $x^3 - 3x^2 + 1$: we have $\rho_i(\beta) \approx 2.879$, so $\beta-2 \in \well_{\rho(\beta)}$. But $N(\beta-2) = -m_{\beta}(2)=3$ so $\beta-2$ is not a unit.
			
			\item $x^3 - 4x^2 + x + 1$: we have $\rho_i(\beta) \approx 3.651$, so $\beta-2 \in \well_{\rho(\beta)}$. But $N(\beta-2) = -m_{\beta}(2)=5$ so $\beta-2$ is not a unit.
			
			\item $x^3 - 4x^2 + 3x+1$: it has $-(\zeta_7+\zeta_7^{-1})+1$ as a root, so $K=\q(\beta)=\q(\zeta_7+\zeta_7^{-1})$.
			
			\item $x^3 - 5x^2 + 4x + 1$: we have $\rho_i(\beta) \approx 3.912$, so $\beta-2 \in \well_{\rho(\beta)}$. But $N(\beta-2) = -m_{\beta}(2) = 3$ so $\beta-2$ is not a unit.
		\end{itemize}
		This completes the proof of Theorem \ref{thm:z-formcubicfield}.
	\end{proof}
	
	To complete the proof of Theorem \ref{thm:z-formcubicfield}, it suffices to prove Lemma \ref{lem:pointsonlineareunits}. Recalling the notations defined in Section \ref{subsec:latticeO_K}, let $H=\{(y_1,y_2,y_3)\in \mathbb{R}^3 | y_1+y_2+y_3=0\}$ be the hyperplane in $\mathbb{R}^3$ orthogonal to $\rho(1)=(1,1,1)$, $p:\mathbb{R}^3\ra H$ the orthogonal projection map onto $H$, and let $\Lambda=p(\rho(\ok{}))$. Let $u=(u_1,u_2,u_3)\in\Lambda\setminus\{0\}$ be a shortest non-zero vector of $\Lambda$, and let $\beta\in\ok{}$ be such that $p(\rho(\beta))=u$. If $\beta\in\z$, then $u=p(\rho(\beta))=0$ which is a contradiction. Hence $\beta\notin \z$. We will show that this $\beta$ satisfies the condition in the lemma. Note that $\ell_{\rho(\beta)}=\ell_u$, and we may show that $u_1, u_2, u_3$ are all distinct using the same argument in the first paragraph of the proof of Theorem \ref{thm:z-formcubicfield}.
	
	We only prove the case when $u_1>u_2>u_3$ since the proofs are similar in other cases. Note that $\ell_{u,1}$ and $\ell_{u,2}$ are contained in the octant $\{(y_1,y_2,y_3)\in\mathbb{R}^3: y_1>0,\ y_2>0,\ y_3<0\}$ and $\{(y_1,y_2,y_3)\in\mathbb{R}^3: y_1>0,\ y_2<0,\ y_3<0\}$ respectively.
	
	\begin{claim}
		There exists $\delta_1 \in \ok{\vee,(1,1,-1)}$ and $\delta_2 \in \ok{\vee,(1,-1,-1)}$ satisfying the following:
		\begin{equation}\label{eqn:conddelta_1}
			\tr(\delta_1\alpha)=1 \quad   \text{for all } \alpha\in \ok{} \text{ with } \rho(\alpha)\in\ell_u,
		\end{equation}
		\begin{equation}\label{eqn:conddelta_2}
			\tr(\delta_2\alpha)=1 \quad   \text{for all } \alpha\in \ok{} \text{ with } \rho(\alpha)\in\ell_u.
		\end{equation}
	\end{claim}
	
	Showing this claim will immediately prove Lemma \ref{lem:pointsonlineareunits} in the following way. If $\alpha \in \ok{}$ satisfies $\rho(\alpha) \in \ell_{\rho(\beta),1} = \ell_{u,1}$, then $\alpha \in \ok{(1,1,-1)}$ and $\tr(\delta_1\alpha) = 1$ for some $\delta_1 \in \ok{\vee,(1,1,-1)}$. Thus by Lemma \ref{lem:genlem4.6ofKY}, every $\alpha\in\ok{}$ with $\rho(\alpha)\in\ell_{u,1}$ should be a unit. A similar argument can be applied to conclude that every $\alpha \in \ok{}$ with $\rho(\alpha) \in \ell_{u,2}$ should be a unit. Therefore we now focus on proving the above claim for the rest of this section.

	Now we show that there exists $\delta_1\in\ok{\vee,(1,1,-1)}$ satisfying \eqref{eqn:conddelta_1} by considering a specific linear functional $g:\Lambda \ra \z$ and taking $\delta_1$ corresponding to $f(g)=g\circ p \circ \rho:\ok{}\ra \z$ via the bijection in Lemma \ref{lem:bij-codiff-linfunctional}. Since $u$ is a primitive vector of $\Lambda$, it can be extended to a basis of $\Lambda$, say $\Lambda=\z u + \z v$ for some $v\in\Lambda$.

	Let $\ell$ be the line in $H$ passing through $v$ and parallel to $u$, and consider the cones $C_0,C_1$, and $C_2$ in $H$ defined by 
	\[
	\begin{aligned}
		C_0&:=\left\{(y_1,y_2,y_3) \in H : y_1>y_2>y_3 \right\},\\
		C_1&:=\left\{(y_1,y_2,y_3) \in H : y_1>y_3>y_2  \right\},\\
		C_2&:=\left\{(y_1,y_2,y_3) \in H : y_3>y_1>y_2  \right\}.
	\end{aligned}
	\]
	Note that $u\in C_0$ since we are assuming $u_1>u_2>u_3$. Also all of $C_0,C_1$, and $C_2$ are cones formed by two rays forming angle of $\frac{\pi}{3}$. Replacing $v$ with $-v$ if necessary, we may assume that the line $\ell$ passes both $C_1$ and $C_2$ (see Figure \ref{fig:planeH}).
	
	\begin{figure}
		\centering
		\begin{tikzpicture}
			\filldraw (0,0) circle (2pt) node[below]{$\mathbf{0}$};
			\filldraw (1.5,-0.5) circle(2pt) node[right]{$u$};
			\filldraw[thin] (-2.1,2.4) node[above]{$\ell$}--(0,1.7)circle(2pt) node[above]{$v$}--(1.5,1.2) circle(2pt) node[right]{$v_1=v+k_1u$}--(3,0.7);
			
			\draw[->] (0,0)--(1.455,1.164);
			\draw[->] (0,0)--(1.455,-0.485);
			\draw [dotted] (0,-1.4) --(0,1.4) node[pos=0.9]{$C_2$};
			\draw [dotted] (-1.212428,-0.7)--(1.212428,0.7) node[pos=0.9]{$C_1$};
			\draw [dotted] (-1.212428,0.7) --(1.212428,-0.7) node[pos=0.9]{$C_0$};
			
			\draw [thin] (-2,0) node[left]{$y_2=y_3$} --(2,0) ;
			\draw [thin] (-1,-1.73204)--(1,1.73204) node[above]{$\hspace{1cm}y_1=y_3$};
			\draw [thin] (-1,1.73204) --(1,-1.73204)
			node[below]{$\hspace{1cm}y_1=y_2$};
		\end{tikzpicture}
		$\hspace{1cm}$
		\begin{tikzpicture}
			\fill[fill=yellow] (0,0)--(1.905244,-1.1)--(0,-2.2);
			\filldraw (0,0) circle (2pt) node[below]{$\mathbf{0}$};
			\filldraw (1.5,1.2) circle(2pt) node[right]{$v$};
			\filldraw (0.470588,-0.588235) circle (2pt) node[above]{$\delta_1$};
			
			\draw[->] (0,0)--(1.455,1.164);
			\draw[->] (0,0)--(0.4470586,-0.55882325);
			\draw [dotted] (0,-1.4) node[below]{$(0,1,-1)$}--(0,1.4);
			\draw [dotted] (-1.212428,-0.7)--(0,0) node[pos=1.9]{$C_1$};
			\draw [dotted] (-1.212428,0.7) --(1.212428,-0.7)node[right]{$(1,0,-1)$};
			
			\draw [thin] (-2,0)--(2,0) node[right]{$(2,-1,-1)$};
			\draw [thin] (-1,-1.73204)--(1,1.73204) node[above]{$\hspace{1cm}(1,-2,1)$};
			\draw [thin] (-1,1.73204)--(0,0)
			node[pos=1.7]{$D_1$};
		\end{tikzpicture}
		\caption{Description of cones, points, and a line on plane $H$}\label{fig:planeH}
	\end{figure}
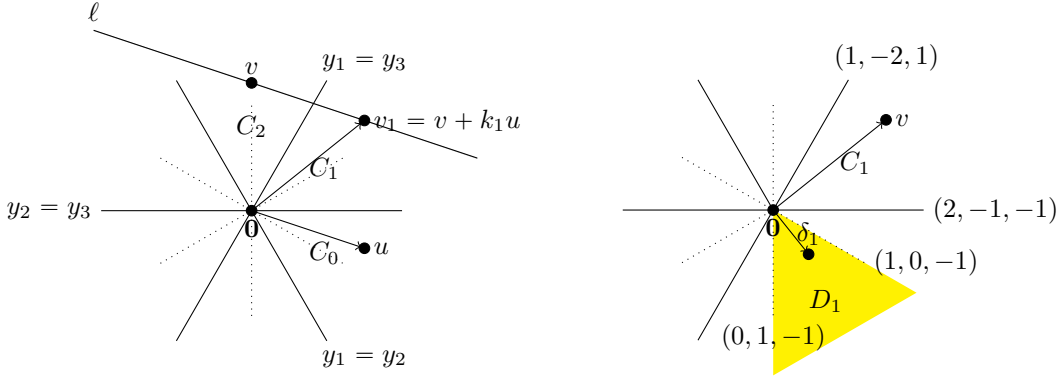
	
	We show that the distance $d_\ell$ from the origin to the line $\ell$ on $H$ is at least $\frac{\sqrt{3}}{2}\lVert u \rVert$. Suppose on the contrary that $d_{\ell} <\frac{\sqrt{3}}{2}\lVert u \rVert$. One can calculate the segment $s:=\ell\cap \{y\in H : \lVert y\rVert < \lVert u \rVert\}$ has length $2\sqrt{\|u\|^2 - d_{\ell}^2}$, so that it is longer than  $\|u\|$. Noting that the vectors of the form $v+ku\in\Lambda$ ($k\in\z$) are placed on $\ell$ in every $\lVert u \rVert$ distance, there should be at least one vector $v+k_0u$ on $s$ for some $k_0\in\z$. Then  by the definition of $s$ we have $\lVert v+k_0u \rVert < \lVert u \rVert$, which contradicts to the fact that $u$ is a non-zero vector of the smallest norm of $\Lambda$. Therefore, we have $d_\ell\ge \frac{\sqrt{3}}{2}\lVert u \rVert$.
	
	Now we claim that the length of the segment $s_1:= \ell \cap C_1$ is at least $\lVert u \rVert$. Denoting $\theta$ the angle formed by the lines $\ell$ and $y_2=y_3$, we can calculate
	\[
	|\ell \cap C_1 | = \left( \cot(\theta) - \cot\left(\theta +\frac{\pi}{3}\right) \right) d_l.
	\]
	Noting that $0<\theta < \frac{\pi}{3}$ and that the function $\cot(\theta)- \cot\left(\theta + \frac{\pi}{3}\right)$ is decreasing in $0\le\theta\le\frac{\pi}{3}$, we have $ |\ell \cap C_1| > \frac{2}{\sqrt{3}} d_l \ge \|u\|$. Therefore, there exists a vector $v_1:=v+k_1u\in\Lambda$ for some $k_1\in \z$ lying on $\ell\cap C_1$.
	
	Hence we have obtained another basis $\{u,v_1\}$ of $\Lambda$ such that $v_1\in C_1$. Consider the linear functional $g_1:\Lambda \ra \z$ defined by $g(u)=1$ and $g(v_1)=0$, and let $\delta_1$ be the element in $\ok{\vee}$ corresponding to $f(g_1)$ via the bijection in Lemma \ref{lem:bij-codiff-linfunctional}, that is, $\delta_1\in\ok{\vee}$ satisfying $f(g_1)(\alpha)=\tr(\delta_1 \alpha)$ for all $\alpha \in \ok{}$. Then $\rho(\delta_1)\in H$ since $\tr(\delta_1)=0$ by Lemma \ref{lem:bij-codiff-linfunctional}, and it satisfies \eqref{eqn:conddelta_1} since for any $\alpha\in\ok{}$ with $\rho(\alpha)\in \ell_u$, we have
	\[
	\tr(\delta_1 \alpha)=f(g_1)(\alpha)=g_1\circ p\circ\rho(\alpha)=g_1(u)=1.
	\]
	Moreover, letting $\alpha_1\in \ok{}$ be such that $p(\rho(\alpha_1))=v_1$, and writing $v_1=\rho(\alpha_1)+t\cdot \rho(1)$ for some $t\in\mathbb{R}$, we have
	\[
	\rho(\delta_1)\cdot v_1=\tr(\delta_1\alpha_1)+t\rho(\delta_1)\cdot\rho(1)=\tr(\delta_1 \alpha_1)+t\tr(\delta_1)=f(g_1)(\alpha_1)+t\cdot0=g_1(v_1)+ 0=0.
	\] 
	Hence $\rho(\delta_1)$ is orthogonal to $v_1$. Similarly, one may show that $\rho(\delta_1)\cdot u=g_1(u)=1>0$. Therefore, $\rho(\delta_1)$ is contained in the cone
	\[
	D_1:=\left\{(y_1,y_2,y_3) \in H : (y_1,y_2,y_3)\cdot (-1,2,-1)>0,\ (y_1,y_2,y_3)\cdot (2,-1,-1)>0  \right\},
	\]
	which can be obtained by rotating the cone $C_1$ by $\frac{\pi}{2}$ on $H$ (see Figure \ref{fig:planeH}). For any $(y_1,y_2,y_3)\in H$, noting that $y_1+y_2+y_3=0$, the conditions defining $D_1$ are equivalent to $3y_2>0$ and $3y_1>0$, and hence $y_3<0$. Therefore, we have $\delta_1\in\ok{\vee,(1,1,-1)}$. 
	
	One may find $\delta_2\in\ok{\vee,(1,-1,-1)}$ satisfying \eqref{eqn:conddelta_2} by following a similar argument by replacing $C_1$ and $D_1$ with $C_2$ and $D_2:=\left\{(y_1,y_2,y_3) \in H : (y_1,y_2,y_3)\cdot (1,-2,1)>0,\ (y_1,y_2,y_3)\cdot (1,1,-2)>0  \right\}$, respectively. This completes the proof of Lemma \ref{lem:pointsonlineareunits}.

	\section{Real biquadratic field case}\label{sec:biquadratic}
	
	In this section, we prove that there is no real biquadratic field $K$ that admits a $\z$-form that is universal over $\ok{}$.
	
	Let $p,q>1$ be two distinct square-free integers, and let $r=\frac{pq}{\gcd(p,q)^2}$. By a real biquadratic field, we mean the field of the form $K=\q(\sqrt{p},\sqrt{q})$. It is of degree $4$ over $\q$ and $(1,\sqrt{p},\sqrt{q},\sqrt{r})$ is a $\q$-basis of $K$. There are four embeddings of $K$ into $\mathbb{C}$, defined for $\alpha=x+y\sqrt{p}+z\sqrt{q}+w\sqrt{r}\in K$ by
	\[
	\begin{aligned}
		\rho_1(\alpha)=x+y\sqrt{p}+z\sqrt{q}+w\sqrt{r},\\
		\rho_2(\alpha)=x-y\sqrt{p}+z\sqrt{q}-w\sqrt{r},\\
		\rho_3(\alpha)=x+y\sqrt{p}-z\sqrt{q}-w\sqrt{r},\\
		\rho_4(\alpha)=x-y\sqrt{p}-z\sqrt{q}+w\sqrt{r}.
	\end{aligned}
	\]
	An integral basis of $\ok{}$ is given by \cite[Theorem 2]{W} according as $(p,q)\Mod{4}$. After suitably interchanging the role of $p,q,$ and $r$, any case can be converted into one of the following five cases (see also \cite[Section 2.2]{KTZ}):
	\[
	\begin{array}{lll}
		\text{Case }1&\left\{ 1, \sqrt{p}, \sqrt{q}, \frac{\sqrt{p}+\sqrt{r}}{2}\right\} &\text{with } (p,q)\equiv(2,3)\Mod{4}\\
		
		\text{Case }2&	\left\{ 1, \sqrt{p}, \frac{1+\sqrt{q}}{2}, \frac{\sqrt{p}+\sqrt{r}}{2}\right\} &\text{with } (p,q)\equiv(2,1)\Mod{4}\\
		
		\text{Case }3&	\left\{ 1, \sqrt{p}, \frac{1+\sqrt{q}}{2}, \frac{\sqrt{p}+\sqrt{r}}{2}\right\}	 &\text{with } (p,q)\equiv(3,1)\Mod{4}\\
		
		\text{Case }4&	\left\{ 1, \frac{1+\sqrt{p}}{2}, \frac{1+\sqrt{q}}{2}, \frac{1+\sqrt{p}+\sqrt{q}+\sqrt{r}}{4}\right\} & \text{with } (p,q)\equiv(1,1)\Mod{4}, \ \gcd(p,q)\equiv 1\Mod{4}\\
		
		\text{Case }5&	\left\{ 1, \frac{1+\sqrt{p}}{2}, \frac{1+\sqrt{q}}{2}, \frac{1-\sqrt{p}+\sqrt{q}+\sqrt{r}}{4}\right\} &\text{with } (p,q)\equiv(1,1)\Mod{4}, \ \gcd(p,q)\equiv 3\Mod{4}\\
	\end{array}
	\]
	
	\begin{lem}\label{lem:basisofcodiff} According to the basis of $\ok{}$ given above, the dual basis of $\ok{\vee}$ is given as follows:
		\[
		\begin{array}{lll}
			\text{Case }1&\left\{ \frac{1}{4}, \frac{1}{4\sqrt{p}}-\frac{1}{4\sqrt{r}}, \frac{1}{4\sqrt{q}}, \frac{1}{2\sqrt{r}}\right\} 
			\\
			
			\text{Case }2\&3 &	\left\{ \frac{1}{4}-\frac{1}{4\sqrt{q}}, \frac{1}{4\sqrt{p}}-\frac{1}{4\sqrt{r}}, \frac{1}{2\sqrt{q}}, \frac{1}{2\sqrt{r}}\right\} 
			\\
			
			\text{Case }4&	\left\{ \frac{1}{4}-\frac{1}{4\sqrt{p}}-\frac{1}{4\sqrt{q}}+\frac{1}{4\sqrt{r}}, \frac{1}{2\sqrt{p}}-\frac{1}{2\sqrt{r}}, \frac{1}{2\sqrt{q}}-\frac{1}{2\sqrt{r}}, \frac{1}{\sqrt{r}}\right\}
			
			\\
			
			\text{Case }5&	\left\{ \frac{1}{4}-\frac{1}{4\sqrt{p}}-\frac{1}{4\sqrt{q}}-\frac{1}{4\sqrt{r}}, \frac{1}{2\sqrt{p}}+\frac{1}{2\sqrt{r}}, \frac{1}{2\sqrt{q}}-\frac{1}{2\sqrt{r}}, \frac{1}{\sqrt{r}}\right\} 
			
			\\
		\end{array}
		\]	
	\end{lem}
	\begin{proof}
		Recall that $\ok{\vee}$ is by definition the dual lattice of $(\ok{},B_{\tr})$. Thus we need only to check that the basis given in the lemma is the dual basis of $\ok{}$ given above for each case, which may easily be verified from straightforward computation of the trace form by noting that $\tr(1)=4$ and $\tr(\sqrt{p}) = \tr(\sqrt{q}) = \tr(\sqrt{r})=0.$
	\end{proof}
	
	\begin{thm}\label{thm:z-formbiquadfield}
		There does not exist a $\z$-form that is universal over a real biquadratic field $K$.
	\end{thm}
	
	\begin{proof}
		Assume to the contrary that there is a real biquadratic field $K=\q\left(\sqrt{p},\sqrt{q}\right)$ that admit a $\z$-form universal over $\ok{}$. We prove the theorem by obtaining a contradiction case-by-case according to the shape of an integral basis of $\ok{}$. General idea is to find explicit $\alpha$ and $\delta$ that satisfy Lemma \ref{lem:genlem4.6ofKY} using our descriptions of basis given in Lemma \ref{lem:basisofcodiff}.\\
		
		\noindent {\bf (Case 1)} $\ok{}=\z 1+\z  \sqrt{p}+\z \sqrt{q}+ \z \left(\frac{\sqrt{p}+\sqrt{r}}{2}\right)$
		
		Consider an element $\alpha=\sqrt{q}\in\ok{}$ and let $\delta=\frac{1}{4\sqrt{q}}$. Noting that $\rho(\sqrt{q})=(\sqrt{q},\sqrt{q},-\sqrt{q},-\sqrt{q})$ and by Lemma \ref{lem:basisofcodiff}, we have
		\[
		\alpha \in \ok{(1,1,-1,-1)}, \ \delta\in \ok{\vee,(1,1,-1,-1)} \ \text{ and } \tr(\delta\alpha)=1.
		\]
		Hence by Lemma \ref{lem:genlem4.6ofKY}, $\alpha$ should be a unit. However, $N(\alpha)=q^2\neq \pm1$, yielding a contradiction.\\
		
		\noindent {\bf (Case 2\&3)} $\ok{}=\z 1 + \z \sqrt{p}+ \z\left( \frac{1+\sqrt{q}}{2}\right)+\z \left(\frac{\sqrt{p}+\sqrt{r}}{2}\right)$.
		
		We have either $p<r$ or $p>r$. If $p<r$, we consider
		\[
		\alpha=\sqrt{p}\in \ok{(1,-1,1,-1)} \text{ and } \delta=\frac{1}{4\sqrt{p}}-\frac{1}{4\sqrt{r}}\in \ok{\vee,(1,-1,1,-1)}.
		\]
		Lemma \ref{lem:basisofcodiff} implies that $\tr(\alpha\delta)=1$. Hence by Lemma \ref{lem:genlem4.6ofKY}, $\alpha$ is a unit. Thus we have $1=N(\alpha)=p^2$, which is impossible. If $p>r$, we consider
		\[
		\alpha'=\sqrt{r}\in \ok{(1,-1,-1,1)},\ \delta'=-\frac{1}{4\sqrt{p}}+\frac{1}{4\sqrt{r}}\in \ok{\vee,(1,-1,-1,1)} \ \text{ and } \tr(\delta'\alpha')=1.
		\]
		Hence by Lemma \ref{lem:genlem4.6ofKY}, $\alpha'$ is a unit. Thus we have $1=N(\alpha')=r^2$, which is impossible.\\
		
		\noindent {\bf (Case 4\&5)} $\ok{}=\z 1 + \z\left( \frac{1+\sqrt{p}}{2}\right)+ \z\left( \frac{1+\sqrt{q}}{2}\right)+\z\left( \frac{1+\mu \sqrt{p}+\sqrt{q}+\sqrt{r}}{2}\right)$ with $\mu \in \{1,-1\}$.
		
		Let $(a,b,c)$ be any permutation of $(p,q,r)$. In both cases, it always holds that
		\[
		\alpha_a = \frac{1+\sqrt{a}}{2} \in \ok{}, \quad \delta_{ab} = \frac{1}{2 \sqrt{a}} - \frac{1}{2\sqrt{b}} \in \ok{\vee} \text{ and }\tr(\alpha_a \delta_{ab})=1.
		\]
		We show that if $a<b$ then $\alpha_a$ and $\delta_{ab}$ have same signature in $K$. We note the following simple lemma.
		\begin{lem}\label{lem:elementsign} Let $u, v \in K$ be nonzero elements satisfying $|\rho_i(u)|>|\rho_i(v)|$ for any embedding $\rho_i$. Then $u+v$ and $u$ have the same signature.
		\end{lem}
		\begin{proof}
			For any embedding $\rho_i$, $\rho_i(u+v) = \rho_i(u) + \rho_i(v)$ has the same sign as $\rho_i(u)$. This shows $\operatorname{Sgn}(u+v)=\operatorname{Sgn}(u)$.
		\end{proof}
		This lemma shows that
		\[
		\operatorname{Sgn}\left(\frac{\sqrt{a}+1}{2}\right) =\operatorname{Sgn}\left(\frac{\sqrt{a}}{2}\right)= \operatorname{Sgn}\left(\frac{1}{2\sqrt{a}} \right)
		=\operatorname{Sgn}\left(\frac{1}{2\sqrt{a}} - \frac{1}{2\sqrt{b}} \right). 
		\]
		So there exists some signature $\epsilon \in \{-1,1\}^4$ such that
		\[
		\alpha_a = \frac{1+\sqrt{a}}{2} \in \ok{\epsilon}, \quad \delta_{ab} = \frac{1}{2 \sqrt{a}} - \frac{1}{2\sqrt{b}} \in \ok{\vee,\epsilon} \text{ and }\tr(\alpha_a \delta_{ab})=1.
		\]
		Applying Lemma \ref{lem:genlem4.6ofKY} shows $\alpha_a$ is a unit, so from $N(\alpha_a) = \left(\frac{1-a}{4}\right)^2 = \pm 1$ we obtain $a=5$. Meanwhile, if we assume $a<b<c$, one can similarly show $\alpha_b$ is also a unit so $b=5$, which is a contradiction to that $a,b,c$ are all distinct.\\
		
		In all cases, we obtain a contradiction. This completes the proof of the theorem.
	\end{proof}

\end{document}